\documentclass{amsart}
\usepackage{amsthm, amscd, amssymb, amsfonts}
\usepackage{mathrsfs}
\usepackage[latin1]{inputenc}
\usepackage{enumerate}
\usepackage{esint}
\usepackage{srcltx}

\usepackage{fancyheadings}
\usepackage{pstricks,pst-node,pst-text,pst-3d,pst-plot}
\usepackage[pdftex]{graphicx}
\usepackage{tikz,pgffor}
\usepackage{textcomp}
\usepackage{amsmath}
\usepackage{latexsym}
\usepackage{eufrak,latexsym}
\usepackage{fancybox,shadow,color,psfrag,float}
\usepackage{ragged2e,comment}
\usepackage{hyperref}

\numberwithin{equation}{section}

\newtheorem{thm}[equation]{Theorem}

\newtheorem{lem}[equation]{Lemma}
\newtheorem{prop}[equation]{Proposition}
\newtheorem{defn}[equation]{Definition}

\theoremstyle{remark}
\newtheorem{rem}[equation]{{\sf Remark}}

\newcommand{\norm}[1]{\left\Vert#1\right\Vert}

\newcommand{\Real}{\mathbb R}

\newcommand{\na}{\mathbb N}
\newcommand{\en}{\mathbb Z}
\newcommand{\set}[1]{\left\{#1\right\}}

\newcommand{\C}{C\,\,}

\newcommand{\di}{\text{d}}

\newcommand{\h}[1]{\hspace*{#1cm}}

\newcommand{\Mb}{M_{\beta}}
\newcommand{\Ms}[1]{M_{(#1/4,#1]}}
\newcommand{\Mbc}{M_{\beta}^c}

\newcommand{\Fb}{\mathcal{F}_{\beta}}
\newcommand{\F}{\mathcal{F}}
\newcommand{\w}[1]{\widetilde{#1}}
\newcommand{\N}{\mathcal{N}}
\newcommand{\W}{\mathcal{W}}

\newcommand{\D}{\mathcal{D}}
\newcommand{\Q}{\mathcal{Q}}
\newcommand{\A}{\mathrm{A}}
\newcommand{\bi}{\bigskip}

\allowdisplaybreaks   % Permite cortar los arreglos de
\title{Two-weight norm inequalities for the Local Maximal Function}

\author{M. Ramseyer, O. Salinas and B. Viviani}



\thanks{This research is partially supported
by grants from Consejo
Nacional de Investigaciones Científicas y
Técnicas (CONICET) and Facultad de Ingenier\'{\i}a
Qu\'{\i}mica Universidad Nacional
del Litoral (UNL), Argentina.}

\address{Instituto de Matemática Aplicada del
Litoral CONICET-UNL and Facultad de Ingenier\'{\i}a
Qu\'{\i}mica, Universidad Nacional del Litoral, Santa Fe, Argentina.}

\email{mramseyer@santafe-conicet.gov.ar}
\email{  salinas@santafe-conicet.gov.ar}
\email{  viviani@santafe-conicet.gov.ar}

\date{}

\begin{document}

\begin{abstract}
For a local maximal function defined on a certain family of cubes lying
``well inside'' of $\Omega$, a proper open subset of $\Real^n$, we characterize
the couple of weights $(u,v)$ for which it is bounded from $L^p(v)$ on
$L^q(u)$.
\end{abstract}

\maketitle

\section{Introduction}

Let $\Omega$ be a proper open and non empty subset of $\Real^n$.
Let $Q=Q(x,l)$ be a cube with sides parallel to the axes.
Here $x$ and $l$ denotes its center and half its side length respectively.
For $0 < \beta < 1$ we consider the family of cubes
well-inside of $\Omega$ defined by
\[  \Fb=\set{Q(x,l): ~ x \in \Omega, ~ l < \beta \, \di(x,\Omega^c)}\,, \]
where, as in all of this work, $\di$ denotes the $d_\infty$ metric.
Related to this family we have the following local maximal function on $\Omega$:
\begin{equation}\label{Maximal Local}
\Mb f(x)=\sup_{x \in Q \in \Fb}\frac 1 {|Q|} \int_Q |f(y)|\,dy \,,
\end{equation}
for every $f \in L^1_{\text{loc}}(\Omega)$ and every $x \in \Omega$.

\bi

In $2014$ E. Harboure and the two last authors (\cite{HSV3}) considered this
operator in the more general setting of a metric spaces $X$ instead of
$\Real^n$ with the Lebesgue measure replaced by a Borel measure $\mu$
defined only on $\Omega$ and doubling on the balls of $\Fb$
(i.e.: $~\mu(B(x,2r)) \leq c\,\mu(B(x,r))$, whenever $B(x,2r) \in \Fb$).
The main result of \cite{HSV3} was a characterization of the weights $w$
such that $\Mb$ is bounded from $L^p(\Omega,w\mu)$ to $L^p(\Omega,w\mu)$,
$1 < p < \infty$, that is there exists a constant $C$ such that
\[  \int_\Omega |\Mb f|^p w\,d\mu \leq C \int_\Omega|f|^p w\,d\mu \,,\]
for every function $f \in L^p(\Omega,w\,d\mu)$.
The classes of weights related to this boundedness are a local version
of the well known $A_p$-Muckenhoupt classes, associated to the
Hardy-Littlewood maximal operator (\cite{Muck}), more precisely non
negative functions $w \in L^1_{\text{loc}}(\Omega,w\,d\mu)$ such that
\[   \Big(  \frac1{\mu(B)} \int_B w\,d\mu \Big) \,\,
\Big(  \frac1{\mu(B)} \int_B w^{-\,\frac1{p-1}}\,d\mu \Big)^{p-1} \leq C_\beta \,,\]
for every ball $B$ in $\Fb$.

After seeing this result, it is natural to ask ourselves about the problem
for a couple of weights $(v,w)$. In connection with it, we should
recall that the situation in the case $\Omega=\Real^n$ do not have an
easy answer. In fact, E. Sawyer (\cite{SAW}) proved that the necessary and
sufficient condition is
\[  \int_Q |M(v^{-\frac1{p-1}}\chi_Q)|^p w\,dx \leq \C \int_Q v^{-\frac1{p-1}} dx\,,\]
for every cube $Q \subset \Real^n$. The problem becomes a little worse
if we want to consider the boundedness from $L^p$ to $L^q$ with
$1 < p \leq q < \infty$. In this case, Sawyer again, but this time as a
particular case of his solution of the problem for fractional maximal
(\cite{SAW2}), showed that the condition turns out to be
\begin{equation}\label{Sawyer condition}
\left( \int_Q |M(v^{-\frac1{p-1}}\chi_Q)|^q w\,dx \right)^{\frac1q} \leq
\C \left( \int_Q v^{-\frac1{p-1}} dx \right)^{\frac1p}\,.
\end{equation}

Our setting is even a bit more complicated since the family $\Fb$
does not include all the balls needed to consider $\Omega$ as a metric
space itself. At this point, if we restrict the problem to the case
$p=q$, a simple application of a result due to B. Jawerth (Theorem
$3.1$, p. $382$ \cite{JAW}) allows us to get

\begin{thm}\label{teo1}
  Given $1< p < \infty$, $0 < \beta < 1$.
  Let $(u,v)$ be a pair of weights.
  Then assuming that $\sigma=v^{1-p'}$ is a weight,
  the following statement are equivalent:
    \begin{equation}\label{union finita 1}
    \Mb : L^p(v) \rightarrow L^p(u)\,,
    \end{equation}
    if and only if there is a constant $c$ such that
    \begin{equation}\label{union finita 2}
        \int_F \Mb (\sigma \chi_F)^p \,u
        \leq \C \int_F \sigma < \infty\,,
    \end{equation}
    for all finite unions $F$ of cubes in $\Fb$,
    $F=\cup_{\text{finite}}Q_i, ~ Q_i \in \Fb$; provided
    \begin{equation}\label{union finita 3}
    M_{\beta,\sigma} : L^p(\sigma) \rightarrow L^p(\sigma)\,.
    \end{equation}
    where
    \[  M_{\beta,\sigma}f(x)= \sup_{x \in Q \in \Fb}
    \frac 1 {\sigma(Q)} \int_Q |f(y)|\,\sigma(y)\,dy\,.\]
\end{thm}

\bi

Leaving aside that we are not getting an answer to the whole problem,
the hypothesis on $(u,v)$ have two drawbacks. In the first place,
integrals over finite unions of cubes must be calculated instead of
only integrals over cubes like in \eqref{Sawyer condition}. In the
second place the conditions involve the operator itself, which
looks worse. The first disadvantage can be overcome by assuming an
extra hypothesis on the weight $v$: a doubling condition related to
$v^{-\frac1{p-1}}$ over balls of $\Fb$. We say that a weight
$u$ satisfies a doubling condition on cubes of $\Fb$,
denoted by $u \in D_\beta$, whenever there exists a constant
$C=C(\beta)$ such that
\[ u(2Q) \leq C_\beta\,u(Q) < \infty \,,\]
for every cube in $\Fb$ such that $2Q \in \Fb$, where $2Q$ means
the concentric cube with side length two times the side length of
$Q$, and $u(Q)=\int_Q u\,dx$.
%Notice that the doubling condition
%on $\Fb$ implies the finiteness of $u(Q)$ for $Q \in \Fb$, but that is
%not necessarily true for all the balls in $\F=\cup_{0<\alpha<1}\F_\alpha$.

By assuming a $\D_\beta$ condition on $v^{-\frac1{p-1}}$, an
application of results in \cite{HSV3} shows that our context
fulfill the hypothesis about the boundedness of $M_{\beta,\sigma}$
in Theorem \ref{teo1}. But taking into account the additional
geometric information we get about the sets on which the maximal
is defined (notice that the Theorem of Jawerth is related to general
basis of open sets in $\Real^n$), a better result can be obtained.
Indeed we can prove the following Theorem.

\begin{thm}\label{teo2}
  Given $1<p\leq q < \infty$, $0 < \beta < 1$.
  Let $(u,v)$ be a pair of weights such that
  $\sigma=v^{1-p'} \in D_\beta$, then
    \begin{equation}\label{uno}
        \left( \int_{\Omega}(\Mb f)^q \,u \right)^{1/q}
        \leq \C \left( \int_{\Omega}|f|^p \,v \right)^{1/p}\,,
    \end{equation}
  for every function $f \in L^p(v)$ if and only if
    \begin{equation}\label{dos}
        \left( \int_Q \Mb (\sigma \chi_Q)^q \,u \right)^{1/q}
        \leq \C \left( \int_Q \sigma \right)^{1/p} < \infty\,,
    \end{equation}
  for every cube $Q \in \Fb$.
\end{thm}

\bi

Note that the hypothesis \eqref{dos}
looks like \eqref{Sawyer condition}. However the appearance
of the operator, the second problem we have mentioned, makes
it difficult to check the condition. In the case $\Omega=\Real^n$ C. Pérez
(Theorem 1.1, \cite{PEREZ}) gave a solution by adding an $A_\infty$-condition
on $v^{-\frac1{p-1}}$. We recall that a weight $u$ belongs to the
$A_\infty$ class of Muckenhoupt if there are positive constants
$c$ and $\delta$ such that
\begin{equation}\label{c-delta}
\frac{u(E)}{u(Q)} \leq c\, \left( \frac{|E|}{|Q|} \right)^\delta\,,
\end{equation}
for every measurable set $E \subset Q$ and every cube $Q$.
With this extra assumption, the necessary and sufficient
condition for the boundedness of the maximal is the existence
of a constant $C$ such that
\begin{equation}\label{Apq condition}
\frac{u(Q)^{\frac pq}\big(v^{-\frac1{p-1}}(Q)\big)^{p-1}}{|Q|^p} \leq C\,,
\end{equation}
for every cube $Q$; which sometimes is referred to as $A_{p,q}$ condition.
It is important to note that we cannot apply the solution given by C.
Pérez because, as it was said before, our setting is not even a metric
space. However, it served as a source of inspiration for our second
result. In order to formulate it we introduce a couple of definitions.

\begin{defn}\label{Ainfinitobeta}
Given $0 < \beta < 1$, we say that a weight $u$ belongs to
$A_\infty^\beta$ if it there are positive constants $c$ and $\delta$
such that \eqref{c-delta} holds for every $Q \in \Fb$.
\end{defn}

\begin{defn}\label{DefiApqbeta}
Let $1 < p \leq q < \infty$ and $0 < \beta < 1$.
We say that the weights $u$ and $v$ lies in the class
$\A_{p,q}^{\beta}$ if and only if
\begin{equation}\label{Apqbeta}
\frac{u(Q)}{|Q|}^{p/q}\,\left( \frac{\sigma(Q)}{|Q|} \right)^{p-1} \leq \C,
\end{equation}
for every cube $Q \in \Fb$, where $\sigma=v^{-\frac1{p-1}}$.
In this cases we write $(u,v) \in \A_{p,q}^{\beta}$.
\end{defn}

\bi

Now we are in position to enunciate our second theorem where the reference
to the operator in the hypothesis on the weights is completely avoided.

\begin{thm}\label{teo3}
Let $p$, $q$, $\beta$ and the weights $u$ and $v$ as in the Theorem above.
In addition if $u \in D_\beta$ and
$\sigma=v^{-1/(p-1)}$ belongs to $A_\infty^\beta$, then
\begin{equation}\label{Max beta}
\Mb : L^p(v) \rightarrow L^q(u)\,;
\end{equation}
if and only if
\begin{equation}\label{Apqbeta1}
(u,v) \in \A_{p,q}^{\beta}~.
\end{equation}
\end{thm}

We note that, under the hypothesis of Theorem \ref{teo3},
the classes $\A_{p,q}^{\beta}$ coincides for different values
of $\beta$. So, as is the one-weight case, we can refer to those
weights as local weights (see Lemma \ref{apqcoinciden} in section $4$.)

As an important tool to prove the theorem above
we consider the centered local maximal function
on $\Omega$, namely $\Mbc$ given by
\begin{equation}\label{Maximal Local}
    \Mbc  f(x)=\mathop{\mathop{\sup}_{Q=Q(x,l)}}_{Q \in \Fb}
    \frac 1 {|Q|} \int_Q |f(y)|\,dy \,,
\end{equation}
for every $f \in L^1_{\text{loc}}(\Omega)$ and every $x \in \Omega$.
For this operator we show that the following theorem holds.
We enunciate it here because it is important itself.

\begin{thm}\label{Mb'}
Let $1 < p \leq q < \infty$, $0 < \beta < 1$ and let $u$ and $v$ be two
weights such that $\sigma=v^{-1/(p-1)}$ belongs to $A_\infty^\beta$. Then
\begin{equation}\label{Max cent}
\Mbc : L^p(v) \rightarrow L^q(u)\,;
\end{equation}
if and only if
\begin{equation}\label{condicion para dos pesos}
(u,v) \in \A_{p,q}^{\beta}~.
\end{equation}
\end{thm}

\bi

\begin{rem}
Although the statements of our theorems are in terms of the maximal
operator we want to remark that minor modifications in the proofs
lead us to corresponding results for a fractional maximal function
defined over $\Fb$.
\end{rem}

The structure of the paper is as follows. Section $2$ contains some
useful geometrical lemmas. The proofs of Theorem \ref{teo2} is in
section $3$. Finally the proofs of Theorems \ref{teo3} and \ref{Mb'}
are in section $4$.

\bi

\section{Technical Lemmas}

In this section we present a covering theorem
and several covering results necessary for the proof of results below.
We will write the following well-known theorem adapted to the context
in our work and without proof.

\begin{thm}[Besicovitch Covering Theorem]\label{BesiTeo}
Let $E \in \Real^n$. For each $x \in E$, let $Q_x$
be a cube centered at $x$. Assume that $E$ is bounded or that
$\sup_{x \in E} l_{Q_x} < 1$. Then, there exists a countable
set $E_0 \subset E$ and a constant $C(n) \in \na$ such that
\begin{equation}\label{BesiTeo1}
E \subset \bigcup_{x \in E_0} Q_x\,;
\end{equation}
\begin{equation}\label{BesiTeo2}
\sum_{x \in E_0} \chi_{Q_x} \leq C(n)\,.
\end{equation}
\end{thm}

\bi

Now, we need to explain the notion of ``cloud'' of a given cube.
That is, given $0<\beta<1$ and a cube
$Q \in \Fb$, we shall denote the set
\begin{equation}\label{cloud}
\N_{\beta}(Q)=\mathop{\mathop{\bigcup}_{R \cap Q \neq \emptyset}}_{R\in\Fb}R\,,
\end{equation}
and we say that these are the ``cloud'' of $Q$.
This idea was introduced in \cite{HSV3} and the proof of the following
lemmas can be found there in the context of the metric spaces.

\begin{lem}\label{lema1}
Let $Q=Q(x,l) \in \Fb$ such that $10\,Q \not\in \Fb$.
We consider $k_0 \in \en$ such that $2^{k_0-1} \leq \di(x,\Omega^c) < 2^{k_0}$.
Then there exists natural numbers $h_1, h_2$ independent of $Q$ such that
\[  2^{k_0-h_1-1} \leq \di(y,\Omega^c) < 2^{k_0+h_2}\,,\qquad \text{for every } \,
y \in \N_\beta(Q)\,.\]
\end{lem}

\begin{proof}
The proof is a consequence of the claim 1 and 2 contained in the proof of
the lemma 2.3 in \cite{HSV3}.
\end{proof}

\bi

Now, denoting by $\D$ the usual family
of dyadic cubes belonging to $\Fb$ we have the following lemma.

\begin{lem}\label{Metodo HSV}
Let $\Omega$ be an open proper subset of $\Real^n$.
Given $0 < \beta < 1$, for each $t \in \na$ such that $2^{-t} \leq \beta/5$,
there exists a covering $\W_t$ of $\Omega$ by dyadic cubes belonging to $\Fb$
and satisfying the following properties
\renewcommand{\theenumi}{\roman{enumi})}
\renewcommand{\labelenumi}{\theenumi}
\begin{enumerate}
\item \label{Metodo HSV - 1}
If $R=R(x_R,l_R) \in \W_t$, then $10\,R \in \Fb$ and
\[  2^{-t-3}\,\di(x_R,\Omega^c) \h{.2} \leq \h{.2} l_R \h{.2}
\leq \h{.2} 2^{-t-1}\,\di(x_R,\Omega^c)\,.\]
\item \label{Metodo HSV - 3}
There is a number $M$, only depending on $\beta$ and $t$, such
that for any cube $Q_0=Q(x_0,l_0) \in \Fb$ with $10 Q_0 \not \in \Fb$,
the cardinal of the set
\[  \W_t(Q_0)=\set{R \in \W_t\,:~R\cap \N_{\beta}(Q_0) \neq \emptyset}\,, \]
is at most $M$. We will call the union of this cubes as
\[  \W_{t,Q_0} = \mathop{\bigcup}_{R \in \W_t(Q_0)} R\,. \]
\end{enumerate}
\end{lem}

\begin{proof}
We will follow the ideas of Lemma $2.3$ in \cite{HSV3}.
So, we only show how we take the covering $\W_t$.
For $k \in \en$, we consider the bands defined by
\[  \Omega_k=\set{x \in \Omega\,: ~ 2^{k-1} \leq \di(x,\Omega^c) < 2^k}\,. \]

If $\Omega_k$ is non empty, let us consider the collection $G_k$ of all
usual dyadic cubes $Q_j=Q(x_j,l_j)$ such that
\[  l_j=2^{k-t-2} \,, \qquad Q_j \cap \Omega_k \neq \emptyset\,,\]
where $t$ is given as in the hypothesis.
It is clear that $\Omega_k \subset G_k$. Moreover, taking
$y \in Q_j$ and $z \in Q_j \cap \Omega_k$ we get
\[  \di(y,\Omega^c) \leq \di(z,\Omega^c) + \di(y,z) \leq 2^k + 2\,l_j
= 2^k + 2^{k-t-1} < 2^{k+1}\,,\]
and
\begin{equation}\label{ecu15}
\di(y,\Omega^c) \geq \di(z,\Omega^c) - \di(y,z) > 2^{k-1} - l_j
= 2^{k-1} - 2^{k-t-2} > 2^{k-2}\,,
\end{equation}
so the inclusion
\begin{equation}\label{ecu17}
Q_j \subset \Omega_{k-1} \cup \Omega_k \cup \Omega_{k+1}\,,
\end{equation}
holds. However there are not cubes intersecting three bands simultaneously. In fact,
suppose that there exists $z, w \in Q_j$ such that $z \in \Omega_{k-1}$
and $w \in \Omega_{k+1}$. Then
\[  2^{k-1} = 2^k - 2^{k-1} \leq \di(w,\Omega^c) - \di(z,\Omega^c)
\leq \di(w,z) \leq 2\,l_j = 2^{k-t-1}\,.\]

This implies that $t$ is less than or equal to $0$ which is a contradiction.
In conclusion, we can say that, for a fixed $k$
there exists in $G_k$ three classes of cubes
\[  Q_j \cap \Omega_{k-1} \neq \emptyset \,, \qquad \text{ó} \qquad
    Q_j \subset \Omega_k \,, \qquad \text{ó} \qquad
    Q_j \cap \Omega_{k+1} \neq \emptyset \,. \]

Next, for each $k$ we define the new collection $E_k$ as follows:
if either $Q_j \subset \Omega_k$ or $Q_j \cap \Omega_{k+1} \neq \emptyset$
we put the cube $Q_j$ in $E_k$. If
$Q_j \cap \Omega_{k-1} \neq \emptyset$ we consider the $2^n$
dyadic sub-cubes and put them in $E_{k-1}$.
So, we note that $E_k$ contains some cubes from $E_{k+1}$
that have been subdivided into $2^n$ sub-cubes.
Thus, the collections $E_k$ are pairwise disjoint and for each
$Q_j(x_j,l_j) \in E_k$ we have that $l_j=2^{k-t-2}$ and
\[  2^{k-1} < \di(x_j,\Omega^c) \leq 2^{k+1} \,. \]

Now, we are able to define a disjoint collection of dyadic cubes by
\begin{equation}\label{ecu24}
\W_t = \bigcup_k E_k \,.
\end{equation}

This is the family of cubes that we will consider. Then, the properties
of the lemma follows by analogous arguments of \cite{HSV3}.
\end{proof}

\bi

\begin{lem}\label{duplica}
Let $0<\beta<1$, $\Omega \subset \Real^n$ and $\mu$
be a measure doubling on $\Fb$. We consider $t \in \en$
such that $2^{-t} \leq \beta/20$ and the covering $\W_t$
of the Lemma above. Then, for any cube $Q$ such that
$10Q \not\in \Fb$ there exists a constant $K$ depending
only on $\beta$ and the constant of the doubling property
of $\mu$ such that
\[  \mu \big( \W_{t,Q} \big) \leq K \,\nu (Q) \,,\]
where $\W_{t,Q}$ is as in Lemma above.
\end{lem}

\begin{proof}
The proof follows the same lines as in Remark $3.2$ in \cite{HSV3}
in the general setting of metric spaces.
\end{proof}

\begin{rem}\label{nubes comparables}
Since $\N_\beta(Q) \subset \W_{t,Q}$
for every cube in $\Fb$, by the Lemma above
we can deduce that
\begin{equation}\label{ecu21}
\mu\big(\N_\beta(Q)\big) \leq \C \mu(Q)\,.
\end{equation}
\end{rem}

\bi

We observe by the construction \eqref{ecu24}
that for each cube $Q_j(x_j,l_j) \in E_k \subset \W_t$ we get
\[  \frac12 \, 2^{-t-2} \leq \frac{l_j}{\di(x_j,\Omega^c)} < 2\,2^{-t-2}\,. \]

In general, we will say that a collection of cubes $\{Q_i\}$ is of Whitney's type if there exists
constants $0 < c_1 < c_2 < 1$ such that
\[   c_1 < l_{Q_i}/\di(x_{Q_i},\Omega^c) < c_2 \,. \]

\begin{lem}\label{solapamiento de nubes controlada}
Let $\{Q_i\}$ be a pairwise disjoint collection of Whitney's type cubes.
Then their clouds have bounded overlapping. More precisely,
there exists a natural number $M>0$ such that
\[  \sum_i \chi_{\N_{\beta}(Q_i)}(x) \h{.2} \leq \h{.2} M\,, \]
for every $x$ in $\Omega$.
\end{lem}

\begin{proof}
Let $\{Q_i\}$ be a such collection, $x_i$ and $l_i$ their centers and
length sides respectively. Take again the bands $\Omega_k$ as in
Lemma \ref{Metodo HSV}. We consider $x \in \Omega_k$ and assume that
\[   x \in \bigcap_{i \in F} \N_\beta(Q_i)\,, \]
for some family of index $F$. Let us prove that there is a
constant $M$ such that the cardinal of
this family is controlled by $M$ for every point $x \in \Omega$.
By lemma \ref{lema1}, if the center $x_i \in \Omega_{k_i}$, we can say that
\[  \N_\beta(Q_i) \subset \bigcup_{j=k_i-h_1}^{k_i+h_2} \Omega_j\,.\]

Thus, the range of $j$ is independent of $Q_i$ and equal to $h=h_2+h_1$.
Now, since $k_i-h_1 \leq k \leq k_i+h_2$ for every $i \in F$ it is easy to see that
\begin{equation}\label{ecu22}
\bigcup_{i \in F} \N_\beta(Q_i) \subset \bigcup_{j=k-h}^{k+h} \Omega_j \,,
\end{equation}
that is, the range of values that may be the union of the clouds is $2h$.

\bi

Now, suppose that there exists $y, z \in \N_\beta(Q_i) \cup \N_\beta(Q_s)$ with $i,s \in F$ and
$l_i \leq l_s$. Let $P_y$, $P_z$, $P_i$ and $P_s$ be cubes such that
\[  y \in P_y\,,\qquad P_y \cap Q_i \neq \emptyset\,, \qquad \qquad
    z \in P_z\,, \qquad P_z \cap Q_s \neq \emptyset\,,\]
\[  Q_i \cap P_i \neq \emptyset \,,\qquad Q_s \cap P_s \neq \emptyset \qquad
\text{and} \qquad x \in P_i \cap P_s\,.\]
Now, we take as in the figure the points
\[  y_i \in P_y \cap Q_i\,,\qquad  z_s \in P_z \cap Q_s\,,\qquad
x_i \in Q_i \cap P_i \qquad \text{and} \qquad   x_s \in Q_s \cap P_s\,. \]
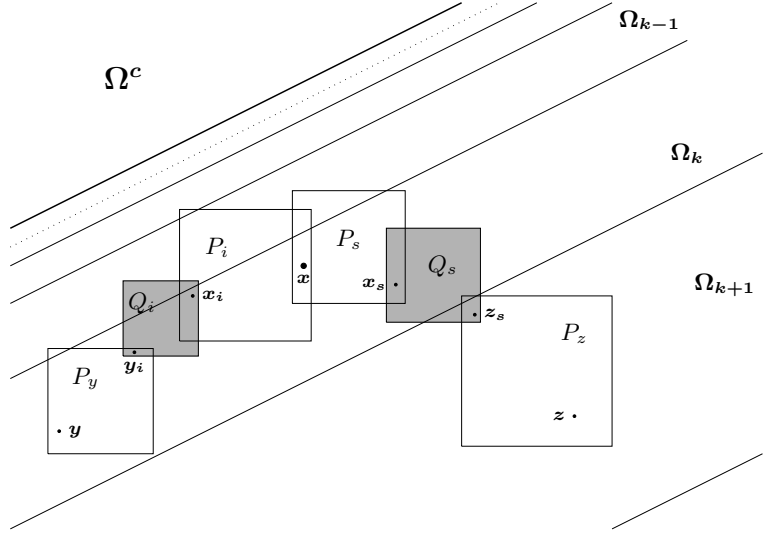
\begin{center}
  \begin{tikzpicture}[scale=0.5]
  \node at (3,12){\Large{$\boldsymbol{\Omega^c}$}};
  \draw [line width=.01 cm](16,0)--(20,2);
  \draw [line width=.01 cm](0,0)--(20,10);
  \draw [line width=.01 cm](0,4)--(18,13);
  \draw [line width=.01 cm](0,6)--(16,14);
  \draw [line width=.01 cm](0,7)--(14,14);
  \draw [line width=.01 cm,dotted](0,7.5)--(13,14);
  \draw [line width=.02 cm](0,8)--(12,14);
  \node at (19,6.5){\small{$\boldsymbol{\Omega_{k+1}}$}};
  \node at (18,10){\small{$\boldsymbol{\Omega_k}$}};
  \node at (17,13.5){\small{$\boldsymbol{\Omega_{k-1}}$}};
  \draw [fill=black](7.8,7) circle (2pt);
  \node [below] at (7.8,7){\scriptsize{$\boldsymbol{x}$}};
  \draw [line width=.01 cm](12,2.2)--(16,2.2)--(16,6.2)--(12,6.2)--(12,2.2);
  \node at (15,5.2){\small{$P_z$}};
  \draw [fill=black,opacity=.3](10,5.5)--(12.5,5.5)--(12.5,8)--(10,8)--(10,5.5);
  \draw [line width=.01 cm](10,5.5)--(12.5,5.5)--(12.5,8)--(10,8)--(10,5.5);
  \node at (11.5,7){\small{$Q_s$}};
  \draw [line width=.01 cm](7.5,6)--(10.5,6)--(10.5,9)--(7.5,9)--(7.5,6);
  \node at (9,7.7){\small{$P_s$}};
  \draw [line width=.01 cm](4.5,5)--(8,5)--(8,8.5)--(4.5,8.5)--(4.5,5);
  \node at (5.5,7.5){\small{$P_i$}};
  \draw [fill=black,opacity=.3](3,4.6)--(5,4.6)--(5,6.6)--(3,6.6)--(3,4.6);
  \draw [line width=.01 cm](3,4.6)--(5,4.6)--(5,6.6)--(3,6.6)--(3,4.6);
  \node at (3.5,6){\small{$Q_i$}};
  \draw [line width=.01 cm](1,2)--(3.8,2)--(3.8,4.8)--(1,4.8)--(1,2);
  \node at (2,4){\small{$P_y$}};
  \draw [fill=black](1.3,2.6) circle (1pt);
  \node [right] at (1.3,2.6){\scriptsize{$\boldsymbol{y}$}};
  \draw [fill=black](15,3) circle (1pt);
  \node [left] at (15,3){\scriptsize{$\boldsymbol{z}$}};
  \draw [fill=black](12.35,5.7) circle (1pt);
  \node [right] at (12.35,5.7){\scriptsize{$\boldsymbol{z_s}$}};
  \draw [fill=black](3.3,4.7) circle (1pt);
  \node [below] at (3.3,4.7){\scriptsize{$\boldsymbol{y_i}$}};
  \draw [fill=black](4.85,6.2) circle (1pt);
  \node [right] at (4.85,6.2){\scriptsize{$\boldsymbol{x_i}$}};
  \draw [fill=black](10.25,6.5) circle (1pt);
  \node [left] at (10.25,6.5){\scriptsize{$\boldsymbol{x_s}$}};
  \end{tikzpicture}
\end{center}

Then, since all the cubes belongs to $\Fb$ and considering \eqref{ecu22}
we have the following estimation
\begin{eqnarray*}
  \di(y,z) & \leq & \di(y,y_i) + \di(y_i,x_i) + \di(x_i,x) + \di(x,x_s) + \di(x_s,z_s) + \di(z_s,z) \\
    & \leq & 6 \, \beta \, 2^{k+h} \,.
\end{eqnarray*}

On the other hand
\[  l_i > c_1 \di(x_i,\Omega^c) > c_1 \, 2^{k-h-1}\,. \]

Thus, there exists at most
\[  \frac{6 \, \beta \, 2^{k+h}}{l_i} \leq \frac{6 \, \beta}{c_1}\, 2^{k+h-k+h+1} = C_\beta\,, \]
disjoint cubes of the family $\{Q_i\}$.
This fact and \eqref{ecu22} say that the family $F$ is finite and then
there exists a fixed natural number $M$, depending only on $\beta$ such that
\[  \sum_i \chi_{\N_{\beta}(Q_i)}(x) \h{.2} \leq \h{.2} M\,, \]
as we wanted to prove.
\end{proof}

\bi

\begin{lem}\label{Metodo C-U}
Let $f$ be a non-negative, locally integrable function
and $\mu$ be a doubling measure on $\Real^n$.
Suppose that for some $h>0$ and some cube $Q=Q(x_Q,l_Q) \in \Fb$
\[  \frac 1 {|Q|} \int_Q f > h\,.\]
\renewcommand{\theenumi}{(\roman{enumi})}
\renewcommand{\labelenumi}{\theenumi}
\begin{enumerate}
\item \label{Metodo C-U 1}
If $10Q \in \Fb$ then there exists a dyadic cube $P=P(x_P,l_P)$ such that
$Q \subset 5P \in \Fb$ and a positive constant $c_1$, independent of $Q$, such that
\begin{equation}\label{tres-a}
\frac 1 {|P|} \int_P f > c_1\,\,h\,.\end{equation}
\item \label{Metodo C-U 2}
If $10Q \not \in \Fb$ then there exists a dyadic cube $R=R(x_R,l_R)$ such that
$Q \subset \W_{t,R}$ and a positive constant $c_2$, independent of $Q$, such that
\begin{equation}\label{tres-b}
\frac 1 {|R|} \int_R f > c_2\,\,h\,,\end{equation}
where $\W_t$ is as in the Lemma \ref{Metodo HSV}.
\end{enumerate}
\end{lem}

\bi

\begin{proof}
Let $Q=Q(x_Q,l_Q)$ be a such cube of the hypothesis.
For \emph{\ref{Metodo C-U 1}} we consider $k \in \en$
such that $2^{k-1} < l_Q \leq 2^k$.
Considering dyadic cubes with side length equal to $2^{k-1}$,
there exists a finite collection of dyadic cubes $P_1,\ldots,P_N$,
with $1 \leq N \leq 3^n$, which intersect the interior of $Q$.
Calling $P$ any of these and taking $z \in Q \cap P$, we have
\[  \di(x_Q,x_P) \leq \di(x_Q,z) + \di(z,x_P) \leq
\frac12 \, l_Q + \frac12 \, l_P \leq 2^{k-1} + 2^{k-2} = \frac32\,l_P\,.\]

Now, if $w \in Q$ we get
\[  \di(w,x_P) \leq \di(w,x_Q) + \di(x_Q,x_P) \leq \frac12 \, l_Q + \frac32\,l_P \leq \frac52\,l_P \,,\]
which implies that $Q \subset 5P$. Moreover, for each $z \in 5P$
\[  \di(z,x_Q) \leq \di(z,x_P) + \di(x_Q,x_P) \leq \frac52 \, l_P + \frac32\,l_P = 4\,l_P\,.\]

Thus, we can deduce that
$Q \subset 5P \subset 8Q$. Now, a simpler estimation
show that $5P \in \Fb$ whenever $10Q$ do it. In fact
\[  l_P < l_Q \leq \frac{\beta}{10} \di(x_Q,\Omega^c) \leq
\frac{\beta}{10} \Big( \di(x_Q,x_P) + \di(x_P,\Omega^c) \Big) =
\frac{\beta}{5} l_P + \frac{\beta}{10} \di(x_P,\Omega^c)\,, \]
then, recalling that $0 < \beta < 1$ we get
\[  \frac12 \, l_P < (1-\frac\beta5) \, l_P < \frac{\beta}{10} \, \di(x_P,\Omega^c)\,,\]
this implies that $5l_P < \beta \, \di(x_P,\Omega^c)$ as required.
Furthermore, for at least one of these dyadic cubes, which we denote by $P$,
\[  \int_P f > \frac{h\,|Q|}{3^n}\,,\]
since otherwise we get a contradiction. In fact
\[  \int_Q f \leq \sum_{t=1}^{N} \int_P f \leq \frac{N\,h\,|Q|}{3^n} \leq h\,|Q|\,.\]

Now, since $5P \subset 8Q$, the Lebesgue measure say that
inequality \eqref{tres-a} follows with $c_1=5^n/24^n$.

\smallskip

In order to prove \emph{\ref{Metodo C-U 2}}, by the Lemma
\ref{Metodo HSV}{\em \ref{Metodo HSV - 3}} the cardinal
of $\W_{t,Q}$ is finite and independent of $Q$,
and its cubes are comparable size with $Q$, the same argument
can be applied to take one of them, namely $R$ such that \eqref{tres-b} holds.
\end{proof}

\bi

\section{Proof of the Results}

\begin{proof}[Proof of the Theorem \ref{teo2}]
Assume that \eqref{uno} holds. In particular,
it is for $f=\sigma \chi_Q$, $Q \in \Fb$. Then
\[  \left( \int_Q \Mb (\sigma \chi_Q)^q \,u \right)^{1/q}
\leq \left( \int_{\Omega} (\sigma \chi_Q)^p \,v \right)^{1/p}
= \left( \int_Q \sigma \right)^{1/p} < \infty.\]

To show that \eqref{dos} implies \eqref{uno}, fix a non negative function
$f \in L^p(\Omega,v)$. By a standard argument, we may assume without loss of
generality that $f$ is bounded and has compact support. Now, for each
$k \in \en$, we consider the sets
\[  A_k=\set{x \in \Omega: 2^k < \Mb f(x) \leq 2^{k+1}}\,.\]
Considering a collection $\{Q_x^k\}_{x \in A_k}$ of cubes such that
\[  \frac 1 {|Q_x^k|} \int_{Q_x^k} |f| > 2^k\,,\]
we define
\[  \Q_1=\set{Q_x^k: ~ 10Q_x^k \in \Fb}\quad \text{and}\quad
\Q_2=\set{Q_x^k: ~ 10Q_x^k \not \in \Fb}\,.\]
For the cubes in $\Q_1$ by \emph{\ref{Metodo C-U 1}} of Lemma
\ref{Metodo C-U} there exists a dyadic cube $P_x^k$ such that $Q_x^k \subset 5P_x^k$,
$5P_x^k \in \Fb$ and
\[  \frac 1 {|P_x^k|} \int_{P_x^k} f > c\,2^k\,.\]
On the other hand, for the cubes in $\Q_2$ we take
$t$ such that $2^{-t} \leq \beta/20$ and consider the covering
$\W_t$ of the Lemma \ref{Metodo HSV}. Now, we
can apply \emph{\ref{Metodo C-U 2}}
of Lemma \ref{Metodo C-U} to have a dyadic cube $R_x^k$
such that its cloud contain the original cube $Q_x^k$ and
\[  \frac 1 {|R_x^k|} \int_{R_x^k} f > c\,2^k\,.\]
Since the $P_x^k$'s and $R_x^k$'s are dyadic and bounded in size
(since $f$ has compact support) we can obtain a maximal disjoint
sub-collection $\set{P_j^k}$ such that for each $x$, either
$Q_x^k \subset 5P_j^k$ or $Q_x^k \subset \W_{t,P_j^k}$ for some $j$.

\smallskip

We define $\w{P}_j^k = 5P_j^k$, if $P_j^k$ was chosen from a cube in $\Q_1$
and $\w{P}_j^k = \W_{t,P_j^k}$ if $P_j^k$ was chosen from a cube in $\Q_2$.
It is clear that $A_k \subset \cup_j \w{P}_j^k$. Now we
define the following sets:
\[ E_1^k = \w{P}_1^k \cap A_k\,, ~ E_2^k = \big( \w{P}_2^k
\backslash \w{P}_1^k \big) \cap A_k \,, ~ \ldots \,,
E_j^k = \big( \w{P}_j^k \backslash \mathop{\cup}_{i=1}^{j-1}
\w{P}_i^k \big) \cap A_k\,, \ldots \]

Thus $A_k = \cup_j E_j^k$ and since the $A_k$'s are disjoint,
the sets $E_j^k$'s are pairwise disjoint for all $j$ and $k$.

In order to prove \eqref{uno} we proceed as follows
\begin{eqnarray*}
  \int_{\Omega}(\Mb f)^q \,u
   & = & \sum_k \int_{A_k}(\Mb f)^q \,u \\
   & = & \sum_{j,k} \int_{E_j^k}(\Mb f)^q \,u \\
   & \leq & \C \sum_{j,k} u(E_j^k)\,2^{kq} \\
   & \leq & \C \sum_{j,k} u(E_j^k)\, \bigg( \frac 1 {|P_j^k|} \int_{P_j^k} f \bigg)^q\,.
\end{eqnarray*}

Now, multiplying and dividing by $\left( \int_{\w{P}_j^k} \sigma \right)^q$ we have
\begin{eqnarray*}
  \int_{\Omega}(\Mb f)^q \,u & \leq &
  \C \sum_{j,k} u(E_j^k)\, \bigg( \frac 1 {|P_j^k|} \int_{\w{P}_j^k} \sigma  \bigg)^q
  \left( \frac{\int_{P_j^k} (f/\sigma)\,\sigma}{\int_{\w{P}_j^k} \sigma} \right)^q\\
  & = & \C \int_X T(f/\sigma)^q\,d\omega \,,
\end{eqnarray*}
where $X = \na \times \en$, the discrete measure $\omega$ on $X$ is given by
\[  \omega(j,k) = u(E_j^k)\, \bigg( \frac1 {|P_j^k|} \int_{\w{P}_j^k} \sigma \bigg)^q \,,\]
and for a non-negative, measurable function $g$, the operator $T$ is defined by
\begin{equation}\label{operator T}
Tg(j,k) = \frac{\int_{P_j^k} g\,\sigma}{\int_{\w{P}_j^k} \sigma}\,.
\end{equation}

By interpolation's theory it is sufficient to show that
$T$ is weak-type $(1,q/p)$ for getting \eqref{uno}.
For this, fix $g$ bounded and with compact support. Then for $\lambda > 0$ we consider
\[  B_\lambda=\set{(j,k)\in X\,:~Tg(j,k) > \lambda}\,.\]

By the definition of $\w{P}_j^k$ we have
\begin{eqnarray*}
 B_\lambda^1 &=& \set{(j,k)\in X\,:~ T g(j,k) > \lambda\,,~\w{P}_j^k=5P_j^k} \,;\\
 B_\lambda^2 &=& \set{(j,k)\in X\,:~ T g(j,k) > \lambda\,,~\w{P}_j^k=\W_{t,P_j^k}} \,.
\end{eqnarray*}

Then, we can estimate
\[  \omega(B_{\lambda}) = \sum_{(j,k)\in B_{\lambda}} u(E_j^k)\,
  \bigg( \frac 1 {|P_j^k|} \int_{\w{P}_j^k} \sigma \bigg)^q
  = \sum_{(j,k)\in B_\lambda^1} \h{.3} + \h{.2}
  \sum_{(j,k)\in B_\lambda^2} = I + II\,. \]

Remembering that $5P_j^k \in \Fb$ and since $E_j^k \subset 5P_j^k$,
it is not difficult to see that
\[  I \leq \sum_{(j,k)\in B_\lambda^1} \int_{E_j^k}
M_\beta(\sigma\chi_{5P_j^k})^q\,u \,. \]

Let now $\{P_i\}$ be the maximal disjoint sub-collection of
$\{P_j^k\,:~(j,k) \in B_\lambda^1\}$. Then, since the $E_j^k$ are pairwise
disjoint and the hypothesis \eqref{dos} we have
\begin{eqnarray*}
  I & \leq & \sum_i \sum_{P_j^k \subset P_i} \int_{E_j^k} M_\beta(\sigma\chi_{5P_j^k})^q\,u \\
    & \leq & \sum_i \int_{5P_i} M_\beta(\sigma\chi_{5P_i})^q\,u \\
    & \leq & \C \sum_i \left( \int_{5P_i} \sigma \right)^{q/p}\,.
\end{eqnarray*}

Finally, by the definition of $B_\lambda^1$, the cubes $P_i$'s are disjoint and $q/p \geq 1$
\begin{eqnarray}
\nonumber I    & \leq & \C \sum_i \left( \frac1\lambda \int_{P_i} g\,\sigma \right)^{q/p} \\
\label{ecu23}  & \leq & \C \left( \frac1\lambda \int_{\Omega} g\,\sigma \right)^{q/p} \,.
\end{eqnarray}

The estimation above follows similar lines of the proof of Theorem $1.1$ in
\cite{CUr}.

\bi

Now, we need to estimate $II$. For this, let $(j,k) \in B_\lambda^2$ and we write
\[  \W_{t,P_j^k} = \mathop{\bigcup}_{m=1}^{t_j^k}P_{j,m}^k \,, \]
where $P_{j,m}^k \in \W_t(P_j^k)$ are disjoint. By Lemma \ref{Metodo HSV}
{\em \ref{Metodo HSV - 3}} we know that $t_j^k \leq M$ where
$M$ is independent of the cubes. Then, considering $t_j^k$ disjoint
sets defined by
\[  E_{j,m}^k = \big( P_{j,m}^k \backslash \mathop{\cup}_{i=1}^{j-1}
\w{P}_i^k \big) \cap A_k\,.\]
So, we get
\[  E_j^k=\bigcup_{m=1}^{t_j^k}E_{j,m}^k \,, \]
where the sets $E_{j,m}^k$ are disjoint in $j$, $m$ and $k$. Then
\begin{eqnarray*}
  II &=& \sum_{(j,k)\in B_\lambda^2} u(E_j^k)\,
  \bigg( \frac 1 {|P_j^k|} \int_{\W_{t,P_j^k}} \sigma \bigg)^q\\
  &=& \sum_{(j,k)\in B_\lambda^2} \sum_{m=1}^{t_j^k}
  \int_{E_{j,m}^k} \bigg( \frac 1 {|P_j^k|} \sum_{l=1}^{t_j^k} \int_{P_{j,l}^k} \sigma \bigg)^q  u\,.
\end{eqnarray*}

Now, we consider for each $P_{j,l}^k$ a finite chain joining $P_{j,l}^k$
with $P_{j,m}^k$, that is, a finite subset of $\W_t(P_j^k)$,
say $R_1,\ldots,R_n$ which are all different, with $R_1=P_{j,l}^k$ and
$R_n=P_{j,m}^k$ and for $R_i$ and $R_{i+1}$ neither
$R_i \subset R_{i+1}$ or $R_{i+1} \subset R_i$.

Moreover, part \emph{\ref{Metodo HSV - 1}} and \emph{\ref{Metodo HSV - 3}} of
the Lemma \ref{Metodo HSV} say that $P_{j,m}^k \in \Fb$ and $n \leq M$.
Thus, since $\sigma$ is doubling on $\Fb$ we can deduce that
$\sigma(P_{j,l}^k)  \leq \C \sigma(P_{j,m}^k)$.
Then, by the Lemma \ref{duplica} again we have
\begin{eqnarray*}
  II
  & \leq & \C \sum_{(j,k)\in B_\lambda^2}
  \sum_{m=1}^{t_j^k} \int_{E_{j,m}^k} \bigg( \frac 1 {|\W_{t,P_j^k}|}
  \sum_{l=1}^{t_j^k} \int_{P_{j,m}^k} \sigma \bigg)^q u \\
   & \leq & \C \sum_{(j,k)\in B_\lambda^2} \sum_{m=1}^{t_j^k}
  \int_{E_{j,m}^k} \bigg( \frac 1 {|P_{j,m}^k|}
  \int_{P_{j,m}^k} \sigma \bigg)^q u \\
   & \leq & \C \sum_{(j,k)\in B_\lambda^2} \sum_{m=1}^{t_j^k}
  \int_{E_{j,m}^k} \Mb \big( \sigma \chi_{P_{j,m}^k} \big)^q u \,,
\end{eqnarray*}
where the last inequality holds because $E_{j,m}^k \subset P_{j,m}^k$.
Let $\set{P_i}$ be a maximal disjoint sub-collection of
$\{P_{j,m}^k\}$ with $1\leq m \leq t_j^k$ and $(j,k) \in B_\lambda^2$.
Then, since the $E_{j,m}^k$'s are pairwise disjoint, we have that
\begin{eqnarray*}
  II & \leq & \C \sum_i \mathop{\sum_{(j,k)\in B_\lambda^2}}_{P_{j,m}^k \subset P_i}
  \int_{E_{j,m}^k} \Mb \big( \sigma \chi_{P_{j,m}^k} \big)^q u \\
  & \leq & \C \sum_i \int_{P_i} \Mb \big( \sigma \chi_{P_i} \big)^q u \,.
\end{eqnarray*}

Now, by inequality \eqref{dos} and the fact that $q \geq p$ we get
\begin{eqnarray*}
  II & \leq & \C \sum_i \left( \int_{P_i} \sigma \right)^{q/p} \\
  & \leq & \C \left( \sum_i \int_{P_i} \sigma \right)^{q/p} \,.
\end{eqnarray*}

Finally, since the operator $T$ is defined on the cubes $P_j^k$ we
need to take again a maximal disjoint sub-collection of the family
$\{P_j^k\}$ with $(j,k) \in B_\lambda^2$. Let $\{P_s\}$ be such sub-collection.
Thus, since the $P_i$'s are disjoint and
$P_i \in \W_{t,P_j^k} \subset \W_{t,P_s}$ for some
$(j,k) \in B_\lambda^2$ and some $s$, by the definition of the operator $T$ we can estimate
\begin{eqnarray*}
II & \leq & \C \left( \sum_s \sum_{i:~P_i \subset \W_{t,P_s}} \int_{P_i} \sigma \right)^{q/p} \\
  & \leq & \C \left( \sum_s \int_{\W_{t,P_s}} \sigma \right)^{q/p} \\
  & \leq & \C \left( \frac 1 {\lambda} \sum_s \int_{P_s} g\,\sigma \right)^{q/p} \\
  & \leq & \C \left( \frac 1 {\lambda} \int_{\Omega} g\,\sigma \right)^{q/p}\,,
\end{eqnarray*}
as we wanted to prove. Then the proof of the Theorem is complete.
\end{proof}

\bi

\section{More manageable conditions on cubes}

Now we concentrate in the classes $\A_{p,q}^{\beta}$.
Since $\F_{\alpha} \subset \F_{\beta}$, whenever $\alpha \leq \beta$
we observe that $\A_{p,q}^{\beta} \subset \A_{p,q}^{\alpha}$.
Moreover, if $\A_{p,q}$ consist in all weights
for what \eqref{Apq condition} holds for every cube $Q \in \Real^n$,
it is clear that $\A_{p,q} \subset \A_{p,q}^{\beta}$.
This inclusion is proper. In fact taking
$u(x)=||x||^\alpha$ and $v(x)=||x||^\gamma$, with
$\gamma = (\alpha + n)\frac pq -n$, it is not difficult to see
that $(u,v) \in \A_{p,q}$ whenever $-n < \alpha < n(q-1)$.
However, if $\Omega= \Real^n - \{0\}$ and $0 < \beta < 1$,
we can check that $(u,v) \in \A_{p,q}^\beta$ for every
power $\alpha \in \Real$.

However, in the next Lemma, we show that, under
certain conditions on the weights the
classes $\A_{p,q}^{\beta}$ really are independent of $\beta$.

\begin{lem}\label{apqcoinciden}
Let $0 < \alpha < \beta < 1$. Let $u$ and $v$ be weights such that
$u,\sigma \in D_\alpha$.
Then \[\A_{p,q}^{\beta} \equiv \A_{p,q}^{\alpha}\,.\]
\end{lem}

\begin{proof}
Let $0 < \alpha < \beta < 1$. By the previous observation,
the Lemma is proved if we show the inclusion
$\A_{p,q}^{\alpha} \subset \A_{p,q}^{\beta}$.
For this, given a cube $Q \in \F_{\beta} \backslash \F_{\alpha}$
we consider the cube $\tilde Q = \frac{\alpha}{\beta}Q$.
Taking $k \in \en$ such that $2^{k-1} < \frac{\beta}{\alpha} \leq 2^k$,
since $\tilde Q \in \F_{\alpha}$, by the doubling condition
on $u$ and $\sigma$ we get
\[ u(Q)^{p/q}\,\sigma(Q)^{p-1} \leq u(2^k \tilde Q)^{p/q}\,\sigma(2^k \tilde Q)^{p-1}
\leq \C u(\tilde Q)^{p/q}\,\sigma(\tilde Q)^{p-1} \leq \C|Q| \,, \]
which proves the lemma.
\end{proof}

\bi

\begin{lem}\label{Ainf-->RHI}
Let $\sigma \in A_{\infty}^{\beta}$. Then $\sigma$ satisfies a Reverse H\"older
inequality, i.e.
\begin{equation}\label{RHI}
\left( \frac{1}{|Q|} \int_Q \sigma^{1+\epsilon} \right)^{1/(1+\epsilon)}
\leq \C \frac{1}{|Q|} \int_Q \sigma \,,
\end{equation}
for every cube $Q \in \Fb$.
\end{lem}

\begin{proof}
We only need to observe that for every
cube $Q=Q(x,l) \in \Fb$ and any cube $\tilde Q = \tilde Q(x',l') \subset Q$ such that
$l = 2 \, l'$ it follow that $\tilde Q \in \Fb$. In fact, since $\di(x,x') \leq l'$
and $\beta < 1$
\[  l = 2\,l' \leq \beta \di(x,\Omega^c)
\leq \beta \di(x,x') + \beta \di(x',\Omega^c)
\leq \beta\,l' + \beta \di(x',\Omega^c)\,,\]
implies
\[  l' \leq \frac{\beta}{2-\beta}\, \di(x',\Omega^c) \leq \beta\, \di(x',\Omega^c)\,. \]
Then, the proof follows a similar way as in \cite{GCRF}.
\end{proof}

\bi

\begin{lem}\label{lema4.2}
Let $0 < \beta < 1$ and we consider $(u,v) \in \A_{p,q}^\beta$.
If $\sigma=v^{-1/(p-1)} \in A_{\infty}^\beta$ then there exists
$\tilde{p} < p$ and $\tilde{q} < q$ such that
$(u,v) \in \A_{\tilde{p},\tilde{q}}^{\beta}$.
\end{lem}

\begin{proof}
Since $\sigma \in A_{\infty}^{\beta}$, by the Lemma
\ref{Ainf-->RHI} it follows that $\sigma$
satisfies a Reverse H\"older inequality as in \eqref{RHI}.
Thus, from the hypothesis on the weights we get
\begin{equation}\label{ecu10}
\frac{u(Q)}{|Q|}^{p/q}\,\left( \frac{1}{|Q|} \int_Q \sigma^{1+\epsilon}
\right)^{(p-1)/(1+\epsilon)} \leq \C.
\end{equation}
Now, taking $\delta > 0$ such that $(p-1)/(1+\epsilon)=p-\delta-1$ the Lemma
is proved defining $\tilde{p}=p-\delta$ and $\tilde{q}=\frac{p-\delta}{p}\,q < q$.
In fact,
\[  \sigma^{1+\epsilon} = v^{-(1+\epsilon)/(p-1)}
 = v^{-1/(\tilde{p}-1)} \,,\]
and noting that $p/q = \tilde{p}/\tilde{q}$ defining
$\tilde{\sigma}=v^{-1/(\tilde{p}-1)}$ by \eqref{ecu10}
we have $(u,v) \in \mathrm{A}_{\tilde{p},\tilde{q}}^{\beta}$.
\end{proof}

\bi

In the next Lemma we show the relation between the local and the centered
local maximal function.

\smallskip

\begin{lem}\label{Max local vs Max centrada}
Let $0 < \alpha < 1/4$. There exists $0 < \gamma < 1$ such that
\begin{equation}\label{Mb vs Mb'}
M_{\alpha}f(x) \leq 2^n M_{\gamma}^cf(x)\,,
\end{equation}
for every locally integrable function $f$ and every $x$ in $\Omega$.
\end{lem}

\begin{proof}
Let $f$ be a locally integrable function.
For $0 < \alpha < 1/4$ and $x \in \Omega$ we consider cubes $Q$, $\tilde Q$
such that $x \in Q \in \F_{\alpha}$ and $\tilde Q$ is centered at $x$ with
$l_{\tilde Q} = 2\,l_Q$.
If we show that $\tilde Q \in \F_{\gamma}$, for some $0 < \gamma < 1$ then
the Lemma will be proved. In fact,
\[  \frac{1}{|Q|}\int_Q|f| \leq \frac{2^n}{|\tilde Q|}\int_{\tilde Q}|f| \leq
2^n M_{\gamma}^c f(x)\,,\]
then, taking the supremum over all cubes in $\F_{\alpha}$ containing $x$
we get \eqref{Mb vs Mb'}. So, we observe that
\[  l_{\tilde Q} \leq 2\,\alpha\,\di (x_Q,\Omega^c)
   \leq 2\,\alpha\,\di (x_Q,x) + 2\,\alpha\,\di (x,\Omega^c)
   \leq \alpha\,l_{\tilde Q} + 2\,\alpha\,\di (x,\Omega^c)\,,\]
thus
\[  l_{\tilde Q} \leq \frac{2\,\alpha}{1 - \alpha}\,\di (x,\Omega^c) \,.\]
Then, it is clear that $\tilde Q \in \F_\gamma$ with
$\gamma = \frac{2\,\alpha}{1 - \alpha} < 1$ since the choose of $\alpha$.
\end{proof}

\bi

\begin{proof}[Proof of the Theorem \ref{Mb'}]
That \eqref{Max cent} implies \eqref{condicion para dos pesos} is trivial
using the test function $\sigma \chi_Q$ for each cube $Q \in \Fb$ and
the definition of $\Mbc$.

On the other hand, since it is clear that
$\Vert \Mbc \Vert_{\infty} \leq \norm{f}_{\infty}$, if we prove
that $\Mbc$ is of weak type $(\tilde p,\tilde q)$ for some number
$\tilde p < p$ and $\tilde q < q$, by applying the Marcinkiewicz
interpolation theorem we will get the result.

In order to do this, let
$U_{\lambda}=\{x \in \Omega:~ \Mbc f(x) > \lambda\}$ and let
\[  \set{Q_x}_{x \in U_{\lambda}} =
\set{Q_x\,:~ Q_x \in \Fb\,,\text{ centered at }x\text{ and }
\frac{1}{|Q_x|}\int_{Q_x}|f| > \lambda} \,,\]
a covering for $U_{\lambda}$.
Then, by the Theorem \ref{BesiTeo} we can select a countable
subfamily of cubes $\set{Q_j}$ which still cover $U_{\lambda}$
and such that $\sum_j \chi_{Q_j}(x) \leq C(n)$.

Then, considering $\tilde p$ and $\tilde q$ provided by the
Lemma \ref{lema4.2} and taking into account the property
of the cubes in the covering we can write
\begin{eqnarray*}
  u(U_{\lambda}) & \leq & u\big(\mathop{\cup}_j Q_j\big) \\
  & \leq & \sum_j \frac{u(Q_j)}{|Q_j|^{\tilde q}}\,|Q_j|^{\tilde q} \\
  & \leq & \frac{C}{\lambda^{\tilde q}} \, \sum_j \frac{u(Q_j)}{|Q_j|^{\tilde q}}\,
  \left( \int_{Q_j}|f| \right)^{\tilde q}\\
  & = & \frac{C}{\lambda^{\tilde q}} \, \sum_j \frac{u(Q_j)}{|Q_j|^{\tilde q}}\,
  \left( \int_{Q_j}|f|\,v^{1/\tilde p}\,\,v^{-1/\tilde p} \right)^{\tilde q}\,.
\end{eqnarray*}

H\"older inequality with $\tilde p > 1$ and Lemma \ref{lema4.2} allows us to get
\begin{eqnarray*}
  u(U_{\lambda}) & \leq &
  \frac{C}{\lambda^{\tilde q}} \, \sum_j \frac{u(Q_j)}{|Q_j|^{\tilde q}}\,
  \left( \int_{Q_j}|f|^{\tilde p}\,v \right)^{\tilde q/\tilde p}
  \left( \int_{Q_j} v^{-1/(\tilde p - 1)} \right)^{(\tilde p - 1)\,
  \tilde q/ \tilde p}\\
  &=& \frac{C}{\lambda^{\tilde q}} \, \sum_j \left\{
  \frac{u(Q_j)}{|Q_j|}^{\tilde p/ \tilde q}
  \left( \frac{\tilde{\sigma}(Q_j)}{|Q_j|} \right)^{\tilde p -1}
  \right\}^{\tilde q/ \tilde p}
  \left( \int_{Q_j}|f|^{\tilde p}\,v \right)^{\tilde q/\tilde p}\\
  & \leq & \frac{C}{\lambda^{\tilde q}}
  \bigg( \sum_j \int_{Q_j}|f|^{\tilde p}\,v \bigg)^{\tilde q/\tilde p}\\
  & \leq & \frac{C}{\lambda^{\tilde q}}
  \bigg( \int_{\Omega}|f|^{\tilde p}\,v \bigg)^{\tilde q/\tilde p}\,,
\end{eqnarray*}
and the proof of the Theorem is complete.
\end{proof}

\bi

Now, we introduce the following maximal function. For each $0 < \beta < 1$ we get
\begin{equation}\label{Maximal Local truncada}
    \Ms{\beta}f(x)=\mathop{\sup}_{x \in Q \in \Fb
    \backslash \F_{\beta/4}} \frac 1 {|Q|} \int_Q |f(y)|\,dy \,.
\end{equation}

\bi

\begin{prop}\label{Acotacion de Max truncada}
Let $0 < \beta < 1$ and let $1 < p \leq q < \infty$. Given two weights
$u$ and $v$ such that $(u,v) \in \A_{p,q}^{\beta}$ and $\sigma \in \D_\beta$,
we have
\begin{equation}\label{Ms acotada}
\Ms{\beta} : L^p(v) \rightarrow L^q(u)\,.
\end{equation}
\end{prop}

\begin{proof}
For each $x \in \Omega$ we choose a
cube $Q_x$ such that $x \in Q_x \in \Fb \backslash \F_{\beta/4}$ and
\[  \Ms{\beta}f(x) \leq \frac 2 {|Q_x|} \int_{Q_x} |f| \,.\]

Now, let $t$ be such that $2^{-t} \leq \beta/20$
and we consider the covering $\W_t$ of $\Omega$ provided
by the Lemma \ref{Metodo HSV}. For simplicity we write
\[  \W_t = \bigcup_j Q_j = \bigcup_j Q_j(x_j,l_j)\,,\]
where the cubes $Q_j$ are disjoint.
Since $x \in Q_j$ for some $j$, denoting $x_Q$ the center
of $Q_x$, we get
\begin{eqnarray*}
  \di(x_Q,\Omega^c) & \leq & \di(x_j,\Omega^c) + \di(x_Q,x) + \di(x_j,x) \\
   & \leq & \di(x_j,\Omega^c) + \beta\,\di(x_Q,\Omega^c) + \beta\,\di(x_j,\Omega^c)\,.
\end{eqnarray*}

This implies that their centers holds
\[  \di(x_j,\Omega^c) \leq \frac{1+\beta}{1-\beta} \di(x_Q,\Omega^c)\,. \]

Since $10Q_j \in \Fb$ by part
{\em \ref{Metodo HSV - 1}} of the Lemma \ref{Metodo HSV}, the inequality
above and the fact $Q_x \not \in \F_{\beta/4}$ we have that
\[  l_j < \frac \beta {10} \di(x_j,\Omega^c)
 \leq \frac \beta {10} \frac {1+\beta} {1-\beta} \di(x_Q,\Omega^c)
 \leq \frac \beta {10} \frac {1+\beta} {1-\beta} \frac 4\beta l_Q
 = \frac 25 \frac {1+\beta} {1-\beta} l_Q = c_\beta\,l_Q\,. \]

Thus, $|Q_j| \leq \C |Q_x|$. Now, it is clear that
$x \in \N_{\beta}(Q_j)$ since $x \in Q_x \cap Q_j$.
Then, by H\"older inequality we can proceed as follows
\begin{eqnarray*}
  \int_{\Omega}\left(\Ms{\beta}f \right)^q \,u
  & \leq & \C \sum_j \int_{Q_j} \frac 1 {|Q_x|^q}
  \bigg( \int_{Q_x} |f(y)|\,dy \bigg)^q \,u \\
  & \leq & \C \sum_j \frac {u(Q_j)} {|Q_j|^q}
  \bigg( \int_{\N_{\beta}(Q_j)} |f| \bigg)^q \\
  & \leq & \C \sum_j \frac {u(Q_j)} {|Q_j|^q} ~
  \sigma\big(\N_{\beta}(Q_j)\big)^{q(p-1)/p}
  \bigg( \int_{\N_{\beta}(Q_j)} |f|^p\,v \bigg)^{q/p}\,. \\
\end{eqnarray*}

Since $\sigma \in \D_\beta$, by Remark \ref{nubes comparables} we get
$\sigma\big(\N_{\beta}(Q_j)\big) \leq \C \sigma(Q_j)$. In addition,
we use the fact that $(u,v) \in \A_{p,q}^{\beta}$, $p \leq q$ and
by the Lemma \ref{solapamiento de nubes controlada} we can conclude that
\begin{eqnarray*}
  \int_{\Omega}\left(\Ms{\beta}f \right)^q \,u
  & \leq & \C \sum_j \bigg( \frac {u(Q_j)^{p/q}} {|Q_j|^p} ~
  \sigma(Q_j)^{p-1}\bigg)^{q/p}\,
  \bigg( \int_{\N_{\beta}(Q_j)} |f|^p\,v \bigg)^{q/p} \\
  & \leq & \C\bigg(\sum_j \int_{\N_{\beta}(Q_j)} |f|^p\,v \bigg)^{q/p} \\
  & \leq & \C\bigg(\int_{\Omega} |f|^p\,v \bigg)^{q/p}\,,
\end{eqnarray*}
which gives \eqref{Ms acotada}.
\end{proof}

\bi

Now, we are going to apply the results mentioned above to prove
the analogous result for the local maximal.

\bi

\begin{proof}[Proof of the Theorem \ref{teo3}]
Clearly \eqref{Max beta} implies \eqref{Apqbeta}.
Conversely let $0<\beta<1$. By Lemma \ref{Max local vs Max centrada}
with $\alpha = \beta/4$ there exists $0 < \gamma < 1$ such that the inequality
\begin{eqnarray} \nonumber\label{ecu11}
M_{\beta}f(x) & \leq & M_{\beta/4}f(x) + M_{(\beta/4,\beta]}f(x)\\
\label{ecu11} & \leq & M^c_{\gamma}f(x) + M_{(\beta/4,\beta]}f(x)\,,
\end{eqnarray}
for every $x \in \Omega$. Then, by Theorem \ref{Mb'} and
Proposition \ref{Acotacion de Max truncada} we have
\[  \int_{\Omega} M_{\beta}f^q\,u
   \leq \C \left( \int_{\Omega} M^c_{\gamma}f^q\,u
  + \int_{\Omega} M_{(\beta/4,\beta]}f^q\,u \right)
  \leq \C \left( \int_{\Omega} f^p\,v \right)^{q/p} \,,\]
and the proof of the Theorem is complete.
\end{proof}

\bi

\bibliographystyle{plain}

\end{document}